\documentclass[pdflatex,sn-basic]{sn-jnl}
\usepackage{pdfsync,bm}
\usepackage{algorithm}
\usepackage{xcolor}
\usepackage{listings}

\jyear{2021}%

\theoremstyle{thmstyleone}%
\newtheorem{theorem}{Theorem}
\newtheorem{proposition}[theorem]{Proposition}%

\theoremstyle{thmstyletwo}%

\theoremstyle{thmstylethree}%

\newtheorem{corollary}{Corollary}

\def \e {{\rm e}}
\def \R{{\mathbb{R}}}
\def \S{{\mathbb{S}}}

\newcommand{\ds}{\displaystyle}
\newcommand{\be}{\begin{equation}}
\newcommand{\ee}{\end{equation}}

\def\p{\partial}
\def\n{\nabla}
\def\nablav{\bm\nabla}
\def\d{\hbox{d}}

\def\vx{{\bf x}}

\def\vu{{\bf u}}
\def\omegav{{\bm\omega}}
\def\vn{{\bf n}}

\def\vx{{\bf x}}
\def\vu{{\bf u}}
\def\vom{{\bm\omega}}
\def\vn{{\bf n}}
\def \SS{{\mathbb{S}^2}}

\def\eps{{\epsilon}}

\def\d{{\partial}}
\def\dd{\mathrm{d}}

\def\bR{\mathbf{R}}

\def \lA {\big\langle \! \! \big\langle}
\def \rA {\big\rangle \! \! \big\rangle}

\begin{document}

\title[Radiative Transfer in a Fluid]{Radiative Transfer in a Fluid}

\author[1]{\fnm{Fran\c cois} \sur{Golse}}\email{francois.golse@polytechnique.edu}

\author*[2]{\fnm{Olivier} \sur{Pironneau}}\email{olivier.pironneau@gmail.com}
\equalcont{\color{red}Submitted to Springer's RACSAM in December 2021\color{black}}

\affil[1]{\orgdiv{CMLS}, \orgname{Ecole polytechnique},\orgaddress{\street{~}\city{Palaiseau}, \postcode{91128}, \state{Cedex}, \country{France}}}

\affil*[2]{\orgdiv{LJLL}, \orgname{Sorbonne Universit\'e}, \orgaddress{\street{~}\city{Paris}, \postcode{75253}, \state{Cedex 5}, \country{France}}}


\abstract{We study the Radiative Transfer equations coupled with the time dependent temperature equation of a  fluid:  existence, uniqueness, a maximum principle are established. A short numerical section illustrates the pros and cons of the method.}




\maketitle

\section*{Introduction}\label{sec1}

In fluid mechanics, \emph{Radiative Transfer} is an important subfield of  \emph{Heat Transfer} with many applications to combustion, micro-wave ovens and climate models.

For the physics of radiative transfer for the atmosphere the reader is sent to \cite{GOO}, \cite{BOH},  to the numerically oriented \cite{ZDU} and to the two mathematically oriented \cite{CHA} and \cite{FOW}.

When Planck's theory of black bodies is used the radiations have a continuum of frequencies governed by the temperature of the emitting body. 

Even when the interactions with the fluid medium are neglected, the radiative transfer equations have 5 spatial dimensions. Hence the problem is numerically quite difficult. 
The stratified approximation is used when the radiation source is far, typically, the Sun.  It is a two dimensional model with one spatial and one angular dimension to which the authors have contributed recently: see  \cite{CBOP} for stratified  radiative transfer alone and \cite{OP}, \cite{FGOP}, \cite{FGOPSinum} for stratified  radiative transfer coupled with the temperature equation in the stationary case.  

In this article the complete 5 dimensional radiative transfer model is studied when coupled with the time dependent temperature equation in the fluid. 

Existence and uniqueness of a solution is well known when the physically constants do not depend on the frequency of the radiating source, the so called grey model (see \cite{BAR}) and \cite{G87}).  In the non-grey case, some results have been obtained by \cite{MER}, \cite{GOL-PER}, et al. The present article extends these studies done in the eighties.

The  radiative transfer system coupled to the Navier-Stokes equations has been studied by \cite{POM} and \cite{MAS} at least. In the later an existence theorem is given when the coefficients depend on the spatial variables but not on the frequencies of the source.

The paper begins with a statement of  the radiative transfer equations in Section \ref{sec:2}. 
In Section \ref{sec:3} an existence result is given. 
In Section \ref{sec:4}, uniqueness is shown.  The proof is complex and relies on an argument given by \cite{MER} and \cite{CRA}. A maximum principle is also shown.
Finally in Section \ref{sec:5} a numerical example is given.

\section{Fundamental equations and  approximations}\label{sec:2}

To find the temperature $T$ in an incompressible fluid exposed to electromagnetic waves it is necessary to solve the Navier-Stokes equations coupled with the Radiative Transfer equations, as explained in \cite{POM}.  It is a complex partial differential system formulated in terms of a time dependent temperature field $T(\vx,t)$ function of the position $\vx$ in the physical domain $\Omega$ and a light intensity field ${I}_\nu(\vx,\omegav,t)$ 
of frequency $\nu$  in each direction $\omegav$:

Given ${I}_\nu,T,\vu,\rho$ at time zero, find ${I}_\nu,T, \vu,p,\rho$, such that for all
$\{\vx,\omegav,t,\nu\}\in\Omega\times\S_2\times(0,\bar T)\times\R^+$,
\begin{equation}
\begin{aligned}
\label{onea}
&\frac1c\p_t {I}_\nu + \omegav\cdot\nablav {I}_\nu+\rho\bar\kappa_\nu a_\nu\left[{I}_\nu-{\frac{1}{4\pi}\int_{\S^2}} p(\omegav,\omegav'){I}_\nu(\omegav'){d}\omega'\right]
\\&
\hskip6cm = \rho\bar\kappa_\nu(1-a_\nu) [B_\nu(T)-{I}_\nu],
\\ 
&\rho c_V(\p_tT+\vu\cdot\n T) -\nabla\cdot(\rho c_P \kappa_T\nabla T )
\\ &
\hskip2cm +\frac{1}c\int_0^\infty\int_{\S^2}{I}_\nu d\mu d\nu 
+ \nablav\cdot\int_0^\infty{\int_{\S^2}} {I}_\nu(\omegav')\omegav {d}\omega {d}\nu =0
\\ 
&\p_t\vu+\vu\cdot\n\vu-\frac{\mu_F}\rho\Delta\vu + \frac1\rho\n p={\bf g},\quad \n\cdot\vu=0,
\quad 
\p_t\rho+\n\cdot(\rho\vu)=0,
\end{aligned}
\end{equation}
where $\nabla,\Delta$ are with respect to $\vx$, 
$\ds B_\nu(T)=\frac{2 \hbar \nu^3}{c^2[{\rm e}^\frac{\hbar\nu}{k T}-1]}$, is the Planck function,
  $\hbar,c,k$ are the Planck constant, the speed of light in the medium and  
the Boltzmann constant. The density of the medium is $\rho$, the pressure is $p$; $c_P,c_V$ are the compressibility of the fluid at constant pressure or volume; in large area these may be altitude/depth dependent. 
The absorption coefficient $\kappa_\nu:=\rho\bar\kappa_\nu$ comes from computations in atomic physics, but for our purpose it is seen as the percentage of light absorbed per unit length. The scattering albedo is  $a_\nu\in(0,1)$ and $\frac1{4\pi}p(\omegav,\omegav')$ is the probability that a ray in direction $\omegav'$ scatters in direction $\omegav$. The constants $\kappa_T$ and $\mu_F$ are the thermal and molecular diffusions; ${\bf g}$ is the gravity. 

Existence of solution for the fluid part of (\ref{onea}) has been established by \cite{PLL}.

\subsection{The Mathematical Problem}

Denote the angular average radiative intensity by 
$
J_\nu(t,\vx) = \frac1{4\pi}\int_{\S^2} {I}_\nu(\omegav ){d}\omega.
$
If $\frac1c$ is neglected in \eqref{onea}, the following holds:
\begin{equation}\label{oneb}
\nablav\cdot\int_0^\infty{\int_{\S^2}} {I}_\nu(\omegav')\omegav {d}\omega {d}\nu
 =  4\pi\int_0^\infty \rho \kappa_\nu(1-a_\nu)\left( B_\nu(T)-J_\nu\right) {d}\nu
\,.
\end{equation}
Consequently we are led to study the well-posedness of the following system for $I_\nu,J_\nu,T$:
\begin{equation}\label{RTHeat}
\left\{
\begin{aligned}
{}&\vom\cdot\nabla I_\nu+\kappa_\nu I_\nu=\kappa_\nu(1-a_\nu)B_\nu(T)+\kappa_\nu a_\nu J_\nu\,,\quad J_\nu:=\int_{\SS}I_\nu\tfrac{\dd\vom}{4\pi}\,,
\\
&\d_tT+\vu\cdot\nabla T-\lambda\Delta T=\int_0^\infty\kappa_\nu(1-a_\nu)(J_\nu-B_\nu(T))\dd\nu\,,
\\
&I_\nu(\vx,\vom)\!=\!Q_\nu(\vx,\vom)\,,\quad\vom\cdot\vn<0\,,\,\,\vx\in\d\Omega\,,\qquad\qquad\frac{\d T}{\d n}\Big\vert _{\d\Omega}=0\,,
\\
&T\big\vert _{t=0}=T_{in}\,.
\end{aligned}
\right.
\end{equation}
Here $\Omega$ is assumed to be a bounded open subset of $\bR^3$ with $C^1$ boundary, and we denote by $\vn$ the outward unit normal field on $\d\Omega$. We further assume that $\nu\mapsto\kappa_\nu$ and $\nu\mapsto a_\nu$ are measurable functions
satisfying
$$
0\le\kappa_m\le\kappa_\nu\le\kappa_M\,,\qquad 0\le a_\nu\le a_M<1\,,\qquad\nu>0\,,~a.e.\, ,
$$
for some positive constants $a_M$ and $\kappa_m<\kappa_M$. Finally, we assume that the fluid velocity field $(t,\vx)\mapsto\vu(t,\vx)$ is smooth on $[0,+\infty)\times\overline{\Omega}$ and satisfies
$$
\nabla\cdot\vu(t,\vx)=0\text{ for }\vx\in\Omega\,,\qquad\vu(t,\vx)=0\text{ for }\vx\in\d\Omega\,,\qquad t\ge 0\,.
$$

\section{Existence}\label{sec:3}

Given a passive parameter $t$, consider the auxiliary problem
$$
\left\{
\begin{aligned}
{}&\vom\cdot\nabla I_\nu(t,\vx,\vom)=\kappa_\nu(S_\nu(t,\vx)-I_\nu(t,\vx,\vom))\,,&&\qquad \vx\in\Omega\,,\,\,\vert \vom\vert =1\,,
\\
&I_\nu(t,\vx,\vom)=Q_\nu(\vx,\vom)\,,&&\qquad\vom\cdot\vn<0\,,
\end{aligned}
\right.
$$
where the source  $S_\nu$ is isotropic, i.e. not a function of $\omegav$. Define the exit time
$$
\tau_{\vx,\vom}=\sup\{s>0\,\text{ s.t. }\vx-s\vom\in\Omega\}\,.
$$
By the method of characteristics
$$
I_\nu(t,\vx,\vom)=\mathbf 1_{\tau_{\vx,\vom}<+\infty}Q_\nu(\vx-\tau_{\vx,\vom}\vom)\e^{-\kappa_\nu\tau_{\vx,\vom}}+\int_0^{\tau_{\vx,\vom}}\e^{-\kappa_\nu s}\kappa_\nu S_\nu(t,\vx-s\vom)\dd s\,.
$$
Averaging in $\vom$, one finds
\begin{equation}\label{JofS}
\begin{aligned}
J_\nu(t,\vx)=\mathcal J[S_\nu](t,\vx):=&\tfrac1{4\pi}\int_\SS\mathbf 1_{\tau_{\vx,\vom}<+\infty}Q_\nu(\vx-\tau_{\vx,\vom}\vom)\e^{-\kappa_\nu\tau_{\vx,\vom}}\dd\vom
\\
&+\tfrac1{4\pi}\int_\SS\int_0^{\tau_{\vx,\vom}}\e^{-\kappa_\nu s}\kappa_\nu S_\nu(t,\vx-s\vom)\dd s\dd\vom\,.
\end{aligned}
\end{equation}
Since $\kappa_\nu>0$, the functional $\mathcal J$ satisfies the following monotonicity property:
$$
\begin{aligned}
S_\nu(t,\vx)\le S'_\nu(t,\vx)\text{ for a.e. }\vx\in\Omega\text{ and }t>0
\\
\implies\mathcal J[S_\nu](t,\vx)\le\mathcal J[S'_\nu](t,\vx)\text{ for a.e. }\vx\in\Omega\text{ and }t>0&\,.
\end{aligned}
$$
In particular, 
$$
\begin{aligned}
0\le Q_\nu(\vx,\vom)\,,\,\,S_\nu(t,\vx)\le B_\nu(T_M)\,,\quad x\in\Omega\,,\,\,\vert \vom\vert =1\,,\,\,\nu,t>0
\\
\implies 0\le\mathcal J[S_\nu](t,\vx)\le B_\nu(T_M)\,,\quad x\in\Omega\,,\,\,\nu,t>0&\,.
\end{aligned}
$$
That $\mathcal J[S_\nu]\ge 0$ is obvious. As for the upper bound, observe that
$$
\begin{aligned}
\mathcal J[B_\nu(T_M)]=&\tfrac1{4\pi}B_\nu(T_M)\int_\SS \e^{-\kappa_\nu\tau_{\vx,\vom}}\dd\vom+\tfrac1{4\pi}B_\nu(T_M)\int_\SS\int_0^{\tau_{\vx,\vom}}\e^{-\kappa_\nu s}\kappa_\nu \dd s\dd\vom
\\
=&\tfrac1{4\pi}B_\nu(T_M)\left[\int_\SS \e^{-\kappa_\nu\tau_{\vx,\vom}}\dd\vom+\int_\SS(1-\e^{-\kappa_\nu\tau_{\vx,\vom}})\dd\vom\right]
=B_\nu(T_M)\,,
\end{aligned}
$$
so that the desired upper bound follows from the monotonicity of $\mathcal J$.

In order to solve the system \eqref{RTHeat}, we consider the iterative scheme detailed in Algorithm \ref{algo:1}, where we have assumed that
$$
0\le T_{in}(\vx)\le T_M\,,\quad 0\le Q_\nu(\vx,\vom)\le B_\nu(T_M)\,,\quad x\in\Omega\,,\,\,\vert \vom\vert =1\,,\,\,\nu>0\,.
$$
\noindent
\begin{algorithm}
\caption{to solve \eqref{RTHeat}.}
\label{algo:1}
\begin{enumerate}
\item {Start from $T^0\equiv 0$ and $J^0_\nu=\mathcal J[0]$;}
\item \FOR{$n=0,1,\dots,N-1$  }
\begin{enumerate}
\item for all $\nu\in(0,\infty)$ and all $\tau\in(0,Z)$,  by
knowing $T^n\equiv T^n(t,\vx)$ and $J^n_\nu\equiv J^n_\nu(t,\vx)$, define with \eqref{JofS}
$$
J^{n+1}=\mathcal J[a_\nu J^n_\nu+(1-a_\nu)B_\nu(T^n)]\,;
$$
\item Define $T^{n+1}$ to be the solution of the semilinear drift-diffusion equation
$$
\left\{
\begin{aligned}
&{}\d_tT^{n+1}+\vu\cdot\nabla T^{n+1}-\lambda\Delta T^{n+1}+\mathcal B(T^{n+1})=\int_0^\infty\kappa_\nu(1-a_\nu)J^{n+1}_\nu\dd\nu\,,\\
&T^{n+1}\big\vert _{t=0}=T_{in}\,,\qquad\qquad\frac{\d T^{n+1}}{\d n}\Big\vert _{\d\Omega}=0\,,\quad\vx\in\Omega\,,\,\,t>0\,,
\end{aligned}
\right.
$$
where
$\ds
\mathcal B(T):=\int_0^\infty\kappa_\nu(1-a_\nu)B_\nu(\min(T_+,T_M))\dd\nu\,.
$
\end{enumerate}

\end{enumerate}
\end{algorithm}

Applying Theorem 10.9 in \cite{Brezis} (see also section 4.7.2 in \cite{LionsMage1}) shows that, for each $q\in L^2(0,T;H^{-1}(\Omega))$, there exists a unique solution to the convection-diffusion problem
$$
\left\{
\begin{aligned}
&{}\d_t\theta+\vu\cdot\nabla \theta-\lambda\Delta \theta=q\,,\quad\vx\in\Omega\,,\,\,t>0\,,
\\
&\theta\big\vert _{t=0}=T_{in}\,,\qquad\qquad\frac{\d\theta}{\d n}\Big\vert _{\d\Omega}=0\,,
\end{aligned}
\right.
$$
of the form
$$
\theta(t,\cdot)=\Sigma(t,0)T_{in}+\int_0^t\Sigma(t,s)q(s,\cdot)\dd s\in L^2(0,T;H^1(\Omega))\cap C(0,T;L^2(\Omega))\,.
$$
With 
$$
q:=\int_0^\infty\kappa_\nu(1-a_\nu)J^{n+1}_\nu\dd\nu-\mathcal B(T^{n+1})\,,
$$
we see that $T^{n+1}$ is a fixed point of the map $\mathcal F$ defined by
$$
\mathcal F(\theta)(t,\cdot)=\Sigma(t,0)T_{in}+\int_0^t\Sigma(t,s)\left(\int_0^\infty\kappa_\nu(1-a_\nu)J^{n+1}_\nu(s,\cdot)\dd\nu-\mathcal B(\theta(s,\cdot))\right)\dd s\,.
$$
Observe that $\mathcal B$ is Lipschitz continuous with
$$
\text{Lip}(\mathcal B)\le\kappa_M(1-a_M)\sup_{0\le\theta\le T_M}\int_0^\infty B'_\nu(\theta)d\nu=4\kappa_M(1-a_M)T_M^3\,,
$$
so that, arguing as in the proof of Theorem 1.2 in chapter 6 of \cite{Pazy} shows that $\mathcal F$ has a unique fixed point. This defines a unique solution $T^{n+1}\in C([0,+\infty);L^2(\Omega))\cap L^2_{loc}(0,\infty;H^1(\Omega))$.

Next, we seek to compare the solutions $T$ and $T'$ of 
$$
\left\{
\begin{aligned}
&{}\d_tT+\vu\cdot\nabla T-\lambda\Delta T+\mathcal B(T)=R\,,\quad\vx\in\Omega\,,\,\,t>0\,,
\\
&T\big\vert _{t=0}=T_{in}\,,\qquad\qquad\frac{\d T}{\d n}\Big\vert _{\d\Omega}=0\,,
\end{aligned}
\right.
$$
and
$$
\left\{
\begin{aligned}
&{}\d_tT'+\vu\cdot\nabla T'-\lambda\Delta T'+\mathcal B(T')=R'\,,\quad\vx\in\Omega\,,\,\,t>0\,,
\\
&T'\big\vert _{t=0}=T'_{in}\,,\qquad\qquad\frac{\d T'}{\d n}\Big\vert _{\d\Omega}=0\,,
\end{aligned}
\right.
$$
under the assumption that $0\le R\le R'$ on $(0,+\infty)\times\Omega$. Proceeding as in the proof of Lemma 6.2 of \cite{FGOPSinum}, we multiply both sides of the equality satisfied by the difference $T-T'$ by $(T-T')_+$:
$$
\begin{aligned}
\d_t\tfrac12(T-T')_+^2+\vu\cdot\nabla\tfrac12(T-T')_+^2-\lambda\Delta\tfrac12(T-T')_+^2+\lambda\vert \nabla(T-T')_+\vert ^2
\\
+(\mathcal B(T)-\mathcal B(T'))(T-T')_+=(R-R')(T-T')_+\le 0&\,.
\end{aligned}
$$
Integrating over $\Omega$, and taking into account the boundary conditions satisfied by $\vu$ and $(T-T')$ shows that
$$
\frac{d}{dt}\int_\Omega\tfrac12(T-T')_+^2(t,\vx)\dd\vx+\lambda\int_\Omega\vert \nabla(T-T')_+\vert ^2(t,\vx)\dd\vx+\int_\Omega(\mathcal B(T)-\mathcal B(T'))(T-T')_+\le 0\,,
$$
since
$$
\int_{\d\Omega}\left(\vu\cdot n(T-T')_+^2-\lambda\frac{\d}{\d n}(T-T')_+^2\right)\dd\sigma=0\,.
$$
Then, $T\mapsto\mathcal B(T)$ is nondecreasing on $\mathbf R$, since $\kappa_\nu(1-a_\nu)\ge 0$ and $T\mapsto B_\nu(\min(T_+,T_M))$ is nondecreasing on $\mathbf R$ for each $\nu>0$. Hence 
$$
(\mathcal B(T)-\mathcal B(T'))(T-T')_+\ge 0
$$
so that
$$
\int_\Omega\tfrac12(T-T')_+^2(t,\vx)\dd\vx\le\int_\Omega\tfrac12(T_{in}-T'_{in})_+^2(\vx)\dd\vx=0\,.
$$
Therefore
$$
T_{in}\le T'_{in}\text{ on }\Omega\text{ and }R\le R'\text{ on }(0,+\infty)\times\Omega\implies T\le T'\text{ on }(0,+\infty)\times\Omega\,.
$$

This comparison argument shows that
$$
0\le J^{n+1}_\nu(t,\vx)\le B_\nu(T_M)\text{ on }(0,+\infty)\times\Omega\implies 0\le T^{n+1}\le T_M\text{ on }(0,+\infty)\times\Omega\,.
$$
By the same token,
$$
J^n_\nu\le J^{n+1}_\nu\text{ on }(0,+\infty)\times\Omega\implies T^n\le T^{n+1}\text{ on }(0,+\infty)\times\Omega\,.
$$
On the other hand, the monotonicity property of the function $\mathcal J$ shows that
$$
\begin{aligned}
T^{n-1}\le T^n\text{ and }J^{n-1}_\nu\le J^n_\nu\implies J_\nu^n=&\mathcal J[a_\nu J^{n-1}_\nu+(1-a_\nu)B_\nu(T^{n-1})]
\\
\le&\mathcal J[a_\nu J^n_\nu+(1-a_\nu)B_\nu(T^n)]=J^{n+1}_\nu
\end{aligned}
$$
on $(0,+\infty)\times\Omega$.
Summarising, we have essentially proved the following result.
\noindent
\begin{theorem} 
Assume that
$$
0\le\kappa_m\le\kappa_\nu\le\kappa_M\,,\qquad 0\le a_\nu\le a_M<1\,,\qquad\nu>0\,,~ a.e.\,,
$$
and pick boundary and initial data satisfying, for some $T_M$,
$$
0\le T_{in}(\vx)\le T_M\,,\quad 0\le Q_\nu(\vx,\vom)\le B_\nu(T_M)\,,\quad x\in\Omega\,,\,\,\vert \vom\vert =1\,,\,\,\nu>0\,.
$$
Then

\smallskip
\noindent
(1) the sequences $T^n$ and $J^n_\nu$ satisfy 
$$
0=T^0\le T^1\le\ldots\le T^n\le T^{n+1}\le\ldots\le T_M\qquad\text{ on }(0,\infty)\times\Omega\,,
$$
and
$$
0\le J^0_\nu\le J^1_\nu\le\ldots\le J^n_\nu\le J^{n+1}_\nu\le\ldots\le B_\nu(T_M)\quad\text{ on }(0,\infty)\times\Omega,
\text{ for all $\nu>0$.}
$$

\noindent
(2) There exist a temperature field $T\in C([0,+\infty);L^2(\Omega))\cap L^2_{loc}(0,\infty;H^1(\Omega))$ and a radiative intensity $I_\nu\in L^\infty((0,\infty)\times\Omega\times\SS\times(0,+\infty))$ satisfying \eqref{RTHeat} in the sense of weak solutions.
\end{theorem}
\begin{proof}
Statement (1) is a consequence of the monotonicity properties established above. It implies statement (2) by the following elementary observations: first, one can pass to the limit by monotone convergence in the expression 
$$
J_\nu^{n+1}=\mathcal J[a_\nu J^n_\nu+(1-a_\nu)B_\nu(T^n)]
$$
to find that
$$
J_\nu=\mathcal J[[a_\nu J_\nu+(1-a_\nu)B_\nu(T)]\,,
$$
since $B_\nu$ is an increasing function for each $\nu>0$. By the method of characteristics, the formula
$$
\begin{aligned}
I_\nu(t,\vx,\vom)=\mathbf 1_{\tau_{\vx,\vom}<+\infty}Q_\nu(\vx-\tau_{\vx,\vom}\vom)\e^{-\kappa_\nu\tau_{\vx,\vom}}
\\
+\int_0^{\tau_{\vx,\vom}}\e^{-\kappa_\nu s}\kappa_\nu(a_\nu J_\nu(t,\vx-s\vom)+(1-a_\nu)B_\nu(T(t,\vx-s\vom)))\dd s
\end{aligned}
$$
defines a solution of the transfer equation in \eqref{RTHeat}.

 Finally, for each $\phi\in C_c([0,\infty);H^1(\Omega))$ such that $\d_t\phi\in L^2((0,\infty)\times\Omega)$, one has
 $$
 \begin{aligned}
 \int_0^\infty\int_\Omega(\lambda\nabla T^n\cdot\nabla\phi-T^n(\d_t\phi+\vu\cdot\nabla\phi))\dd\vx\dd t=\int_\Omega T_{in}(\vx)\phi(0,\vx)\dd\vx
 \\
 \int_0^\infty\int_0^\infty\int_\Omega\kappa_\nu(1-a_\nu)(J^n_\nu-B_\nu(T^n))\phi\dd\vx\dd\nu\dd t&\,.
 \end{aligned}
 $$
 One can pass to the limit by dominated convergence in all the terms except the one involving $\nabla T^n$. This last term is mastered by the energy balance for the convection-diffusion equation:
 $$
 \begin{aligned}
 \tfrac12\int_\Omega T^n(t,\vx)^2\dd x+\lambda\int_0^t\int_\Omega\vert \nabla T^n(t,\vx)\vert ^2\dd\vx=\tfrac12\int_\Omega T_{in}(\vx)^2\dd x 
 \\
 +\int_0^t\int_0^\infty\int_\Omega\kappa_\nu(1-a_\nu)(J^n_\nu-B_\nu(T^n))T^n\dd\vx\dd\nu\dd s\le t\vert \Omega\vert \mathcal B(T_M)&\,,
\end{aligned}
$$
which implies that $T^n$ is bounded in $L^2_{loc}(0,\infty;H^1(\Omega))$. Since we already know that $T^n\to T$ in $L^p((0,\tau)\times\Omega)$ for all $p\in[1,\infty)$ as $n\to\infty$, we conclude that $T^n\to T$ weakly in $L^2_{loc}(0,\infty;H^1(\Omega))$.
With this last piece of information, we pass to the limit in the weak formulation of the convection-diffusion equation and conclude that
 $$
 \begin{aligned}
 \int_0^\infty\int_\Omega(\lambda\nabla T\cdot\nabla\phi-T(\d_t\phi+\vu\cdot\nabla\phi))\dd\vx\dd t=\int_\Omega T_{in}(\vx)\phi(0,\vx)\dd\vx
 \\
 \int_0^\infty\int_0^\infty\int_\Omega\kappa_\nu(1-a_\nu)(J_\nu-B_\nu(T))\phi\dd\vx\dd\nu\dd t&\,.
 \end{aligned}
 $$
 In other words, $T$ is a weak solution of the second equation in \eqref{RTHeat}. This concludes the proof.

\end{proof}

\section{Uniqueness and maximum principle}\label{sec:4}

For each $\eps>0$ let $s_\eps\in C^\infty(\bR)$ be such that
$$
s_\eps(z)=0\text{ for }z\le 0\,,\quad s_\eps(z)=1\text{ for }z\ge \eps\,,\quad 0\le\eps s'_\eps(z)\le 2\text{ for }z\in\bR\,,
$$
and let 
$
s_+(z)=\mathbf 1_{z>0}\,.
$
Set
$$
S_\eps(y)=\int_0^ys_\eps(z)dz\,.
$$
Henceforth, we use the notation
$$
\lA\phi\rA:=\tfrac1{4\pi}\int_0^\infty\int_{\SS}\phi(\vom,\nu)\dd\vom\dd\nu
$$
Let $(I_\nu,T)$ and $(I'_\nu,T')$ be two solutions of the system above; then
$$
\nabla\cdot\lA\vom(I_\nu-I'_\nu)_+\rA+D_1+D_2=0\,,
$$
where
$$
D_1:=\lA\kappa_\nu(1-a_\nu)((I_\nu-I'_\nu)-(B_\nu(T)-B_\nu(T')))s_+(I_\nu-I'_\nu)\rA\,,
$$
and 
$$
D_2:=\lA\kappa_\nu a_\nu((I_\nu-I'_\nu)-(J_\nu-J'_\nu))s_+(I_\nu-I'_\nu)\rA\,.
$$
Since
$$
\int_{\SS}((I_\nu-I'_\nu)-(J_\nu-J'_\nu))\dd\vom=0,
$$
one has
$$
D_2:=\lA\kappa_\nu a_\nu((I_\nu-I'_\nu)-(J_\nu-J'_\nu))(s_+(I_\nu-I'_\nu)-s_+(J_\nu-J'_\nu))\rA\ge 0\,,
$$
since $z\mapsto s_+(z)$ is nondecreasing, so that
$$
((I_\nu-I'_\nu)-(J_\nu-J'_\nu))(s_+(I_\nu-I'_\nu)-s_+(J_\nu-J'_\nu))\ge 0\,.
$$

On the other hand
$$
D_1=D_3^\eps+s_\eps(T-T')\lA\kappa_\nu(1-a_\nu)((I_\nu-I'_\nu)-(B_\nu(T)-B_\nu(T')))\rA
$$
where
$$
D_3^\eps=\lA\kappa_\nu(1-a_\nu)((I_\nu-I'_\nu)-(B_\nu(T)-B_\nu(T')))(s_+(I_\nu-I'_\nu)-s_\eps(T-T'))\rA\,,
$$
while
$$
\begin{aligned}
\d_tS_\eps(T-T')+\vu\cdot\nabla S_\eps(T-T')-\lambda\Delta(T-T')s_\eps(T-T')
\\
=s_\eps(T-T')\int_0^\infty\kappa_\nu(1-a_\nu)((J_\nu-J'_\nu)-(B_\nu(T)-B_\nu(T')))\dd\nu
\\
=4\pi s_\eps(T-T')\lA\kappa_\nu(1-a_\nu)((J_\nu-J'_\nu)-(B_\nu(T)-B_\nu(T'))\rA\,.
\end{aligned}
$$
Thus
$$
\begin{aligned}
4\pi\nabla\cdot\lA\vom(I_\nu-I'_\nu)_+\rA+\d_tS_\eps(T-T')+\vu\cdot\nabla S_\eps(T-T')-\lambda\Delta(T-T')s_\eps(T-T')
\\
+4\pi(D_3^\eps+D_2)=0
\end{aligned}
$$
Then we integrate both sides of this equality on $\Omega$:
$$
\begin{aligned}
\frac{d}{dt}\int_\Omega S_\eps(T-T')d\vx+4\pi\int_{\d\Omega}\lA\vom\cdot\vn(I_\nu-I'_\nu)_+\rA\dd\sigma(\vx)+\int_{\d\Omega}S_\eps(T-T')\vu\cdot\vn\dd\sigma(\vx)
\\
+\lambda\int_\Omega\vert \nabla(T-T')\vert ^2s'_\eps(T-T')\dd\vx-\lambda\int_{\d\Omega}s_\eps(T-T')\frac{\d(T-T')}{\d n}\dd\sigma(\vx)\\
\\
=-4\pi\int_\Omega(D_3^\eps+D_2)\dd\vx\,.
\end{aligned}
$$
Using the boundary conditions, specifically that
$$
I_\nu(\vx,\vom)\!=\!Q_\nu(\vx,\vom)\text{ and }I'_\nu(\vx,\vom)\!=\!Q'_\nu(\vx,\vom)\,,\quad\vom\cdot\vn<0\,,
$$
with
$$
(Q_\nu-Q'_\nu)(\vx,\vom)\le 0\,,\quad\vom\cdot\vn<0\,,
$$
implies 
$$
\begin{aligned}
\int_{\d\Omega}\lA\vom\cdot\vn(I_\nu-I'_\nu)_+\rA\dd\sigma(\vx)=\int_{\d\Omega}\lA(\vom\cdot\vn)_+(I_\nu-I'_\nu)_+\rA\dd\sigma(\vx)\ge 0,
\\
\int_{\d\Omega}S_\eps(T-T')\vu\cdot\vn\dd\sigma(\vx)=\int_{\d\Omega}s_\eps(T-T')\frac{\d(T-T')}{\d n}\dd\sigma(\vx)=0.
\end{aligned}
$$
Hence
$$
\begin{aligned}
\int_\Omega S_\eps(T-T')(t,\vx)d\vx+4\pi\int_0^t\int_{\d\Omega}\lA(\vom\cdot\vn)_+(I_\nu-I'_\nu)_+\rA(\tau,\vx)\dd\sigma(\vx)\dd\tau
\\
+\lambda\int_0^t\int_\Omega\vert \nabla(T-T')(\tau,\vx)\vert ^2s'_\eps(T-T')(\tau,\vx)\dd\vx\dd\tau+4\pi\int_0^t\int_\Omega(D_3^\eps+D_2)(\tau,\vx)\dd\vx\dd\tau
\\
=\int_\Omega S_\eps(T-T')(0,\vx)d\vx=0
\end{aligned}
$$
under the assumption that
$$
T\Big\vert _{t=0}=T_{in}\quad\text{ and }T'\Big\vert _{t=0}=T'_{in}\quad\text{ with }T_{in}\le T'_{in}\,.
$$

Assume that 
$$
\kappa_\nu(1-a_\nu)(I_\nu+I'_\nu+B_\nu(T)+B_\nu(T'))\in L^1([0,t]\times\Omega\times\SS\times(0,+\infty));
$$
by dominated convergence
$$
\int_0^t\int_\Omega D_3^\eps(\tau,\vx)\dd\vx\dd\tau\to\int_0^t\int_\Omega D_3(\tau,\vx)\dd\vx\dd\tau
$$
where
$$
D_3=\lA\kappa_\nu(1-a_\nu)((I_\nu-I'_\nu)-(B_\nu(T)-B_\nu(T')))(s_+(I_\nu-I'_\nu)-s_+(T-T'))\rA\ge 0
$$
since $z\mapsto s_+(z)$ is nondecreasing and
$$
s_+(T-T')=s_+(B_\nu(T)-b_\nu(T'))
$$
because $B_\nu$ is increasing for each $\nu>0$, so that
$$
((I_\nu-I'_\nu)-(B_\nu(T)-B_\nu(T')))(s_+(I_\nu-I'_\nu)-s_+(T-T'))\ge 0\,.
$$
By Fatou's lemma
$$
\int_\Omega S_\eps(T-T')(t,\vx)d\vx\to\int_\Omega(T-T')_+(t,\vx)d\vx
$$
for a.e. $t\ge 0$, so that
$$
\begin{aligned}
\int_\Omega(T-T')_+(t,\vx)d\vx+4\pi\int_0^t\int_{\d\Omega}\lA(\vom\cdot\vn)_+(I_\nu-I'_\nu)_+\rA(\tau,\vx)\dd\sigma(\vx)\dd\tau
\\
+4\pi\int_0^t\int_\Omega(D_3+D_2)(\tau,\vx)\dd\vx\dd\tau\le 0\,,
\end{aligned}
$$
since
$$
\varliminf_{\eps\to 0}\int_0^t\int_\Omega\vert \nabla(T-T')(\tau,\vx)\vert ^2s'_\eps(T-T')(\tau,\vx)\dd\vx\dd\tau\ge 0\,.
$$
Since all the terms on the left hand side of the previous equality are nonnegative, one has
$$
\begin{aligned}
\int_\Omega(T-T')_+(t,\vx)d\vx=&4\pi\int_0^t\int_{\d\Omega}\lA(\vom\cdot\vn)_+(I_\nu-I'_\nu)_+\rA(\tau,\vx)\dd\sigma(\vx)\dd\tau
\\
=&4\pi\int_0^t\int_\Omega(D_3+D_2)(\tau,\vx)\dd\vx\dd\tau=0\quad\text{ for a.e. }t>0\,.
\end{aligned}
$$
Once it is known that $T\le T'$ a.e. on $(0,+\infty)\times\Omega$, one has
$$
\begin{aligned}
\int_\Omega\lA\kappa_\nu(1-a_\nu)(I_\nu-I'_\nu)_+\rA\dd\vx+\int_{\d\Omega}\lA(\vom\cdot\vn)_+(I_\nu-I'_\nu)_+\rA(\tau,\vx)\dd\sigma(\vx)+\int_\Omega D_2\dd\vx
\\
=\int_\Omega\lA\kappa_\nu(1-a_\nu)(B_\nu(T)-B_\nu(T'))(t,\vx)s_+(I_\nu-I'_\nu)\rA\dd\vx\le 0
\end{aligned}
$$
since $B_\nu(T)-B_\nu(T')\le 0$ while $s_+(I_\nu-I'_\nu)\ge 0$, so that $I_\nu\le I'_\nu$ a.e. on $(0,+\infty)\times\Omega\times\SS\times(0,+\infty)$.

Summarising, we have proved the following

\begin{theorem}
Let $(I_\nu,T)$ and $(I'_\nu,T')$ be two solutions of \eqref{RTHeat} such that
$$
\kappa_\nu(1-a_\nu)(I_\nu+I'_\nu+B_\nu(T)+B_\nu(T'))\in L^1([0,t]\times\Omega\times\SS\times(0,+\infty))
$$
for all $t>0$. Assume that
$$
T\Big\vert _{t=0}=T_{in}\quad\text{ and }T'\Big\vert _{t=0}=T'_{in}\quad\text{ with }T_{in}\le T'_{in}\,,
$$
and that, when $\vx\in\partial\Omega$,
$$
I'_\nu(\vx,\vom)\!=\!Q'_\nu(\vx,\vom) \text{ and } I_\nu(\vx,\vom)\!=\!Q_\nu(\vx,\vom)\le Q'_\nu(\vx,\vom)\,,\quad\vom\cdot\vn<0\,.
$$
Then
$$
I_\nu\le I'_\nu\text{ and }T\le T'\,.
$$
If $T'_{in}=T_{in}$ and $Q_\nu=Q'_\nu$, exchanging the roles of $(I_\nu,T)$ and $(I'_\nu,T')$ in the theorem above leads to the following uniqueness result.
\end{theorem}

\begin{corollary}
There is at most one solution $(I_\nu,T)$ of \eqref{RTHeat} such that
$$
\kappa_\nu(1-a_\nu)(I_\nu+B_\nu(T))\in L^1([0,t]\times\Omega\times\SS\times(0,+\infty))
\quad
\text{for all $t>0$}.
$$
\end{corollary}
In the case where $(I'_\nu,T)$ is a Planck equilibrium, i.e. $I'_\nu=B_\nu(T')$ with $T'=$constant, one obtains Mercier's maximum principle:

\begin{corollary}
If $0\le Q_\nu\le B_\nu(T_M)$ and $0\le T_{in}\le T_M$ and $\Omega$ has finite volume, the solution $(I_\nu,T)$ of \eqref{RTHeat} such that
$$
\kappa_\nu(1-a_\nu)(I_\nu+B_\nu(T))\in L^1([0,t]\times\Omega\times\SS\times(0,+\infty))
$$
satisfies
$$
0\le I_\nu(t,\vx,\vom)\le B_\nu(T_M)\quad\text{ and }\quad 0\le T(t,\vx)\le T_M
$$
for a.e. $(t,\vx,\vom,\nu)$ in $(0,+\infty)\times\Omega\times\SS\times(0,+\infty)$.
\end{corollary}

\section{A Numerical Scheme}\label{sec:5}
We begin with an important observation for the numerical implementation:
\begin{proposition}
Equation \eqref{JofS} can be written as
\begin{equation}\label{convol}
\tilde J_\nu(\vx) = Y_{\kappa_\nu}(\vx)\star \tilde S_\nu(\vx) + \tilde S_\nu^E(\vx),
\text{ with } Y_{\kappa_\nu}(\vx) = \kappa_\nu\frac{\e^{-\kappa_\nu\vert\vx\vert}}{\pi\vert 2\vx\vert^{d-1}}, ~d=2,3.
\end{equation}
where $\star$ denotes a convolution and tildes indicate an extension by zero outside $\Omega$ and
\begin{equation}\label{SE}
S_\nu^E(\vx) = \frac1{2^{d-1}\pi}\int_{\vert\vom\vert=1}{\bf 1}_{\{\tau_{\vx,\vom}<+\infty\}}Q_\nu(\vx-s\vom)\e^{-\kappa_\nu\tau_{\vx,\vom}}d\omega
\end{equation}
\end{proposition}
\begin{proof}
This is because, by integration in spherical coordinates with $\vert \vx \vert=s$, 
\[
\int_{\vert\vom\vert=1}\int_0^\infty \kappa_\nu\tilde S_\nu(\vx-s\vom)\e^{-\kappa_\nu s}d s d\omega
= \int_{\R^{d}}\kappa_\nu\tilde S_\nu(\vx-\vx')\e^{-\kappa_\nu \vert\vx'\vert}\frac{d x'}{\vert x'\vert^{d-1}}
\]
\end{proof}
Notice that the Fourier transform of $Y_{\kappa_\nu}$ satisfies 
\[
{\cal F}Y_{\kappa_\nu}(\xi)={\cal F}Y_1\left(\frac{\xi}{\kappa_\nu}\right)=\frac{\vert\xi\vert}{2\pi\kappa_\nu}\arctan\frac{\vert\xi\vert}{\kappa_\nu}\, .
\]
The numerical implementation is detailed in Algorithm \ref{algo:2}.
\begin{algorithm}\caption{To solve \eqref{convol} with $S_\nu =  a_\nu J_\nu+(1-a_\nu)B_\nu(T)$ }\label{algo:2}
\textbf{for} each $\nu>0$,
\begin{enumerate}
\item Compute $\vx\mapsto \tilde S^E_\nu(\vx)$ by \eqref{SE} and ${\cal F}Y_{\kappa_\nu}=\frac{\kappa_\nu}{2\pi\vert\xi\vert}\arctan(\kappa_\nu\vert\xi\vert)$. 
\item \textbf{for} n=0,1,..,N
\begin{enumerate}
\item Compute the Fourier transforms  ${\cal F}\tilde S_\nu$ .
\item Compute $Y_{\kappa_\nu}\star S_{\nu} = {\cal F}^{-1}({\cal F}Y_{\kappa_\nu}\cdot {\cal F}\tilde S_{\nu})$.
\item Set $J^{n+1}_\nu(\vx) = S^E_\nu(\vx) + {\cal F}^{-1}({\cal F}Y_{\kappa_\nu}\cdot {\cal F}S_{\nu})$.
\item Compute $T^{n+1}$  solution of 
$$\int_0^\infty\kappa_\nu(1-a_\nu)B_\nu(T)d\nu=\int_0^\infty\kappa_\nu(1-a_\nu)J_\nu^{n+1}d\nu$$.
\end{enumerate}
\end{enumerate}
\end{algorithm}

\subsection{A Bidimensional Example}
\begin{figure}[hbt]
\centering
\includegraphics[width=8cm]{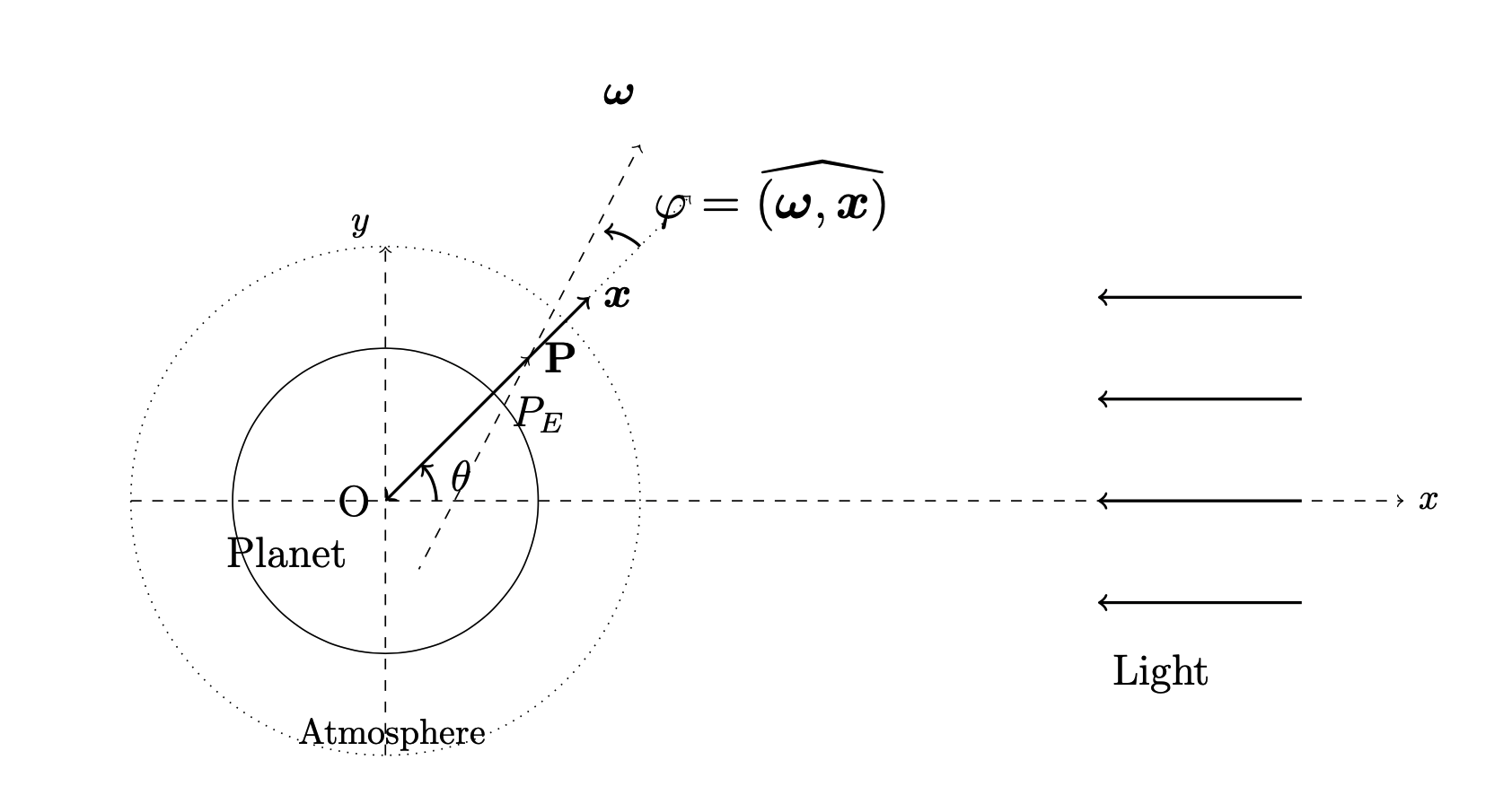}
\caption{\label{fig:tikz}{\it The light source, in the far right, sends lightwaves to the planet; it is assumed that the light is unaffected by the atmosphere.  Hence  point P in the atmosphere receives only the radiations emitted by the planet . A cross section of the planet and its atmosphere is shown in the plane defined by the axis $Ox$ and the point $P$. The light intensity in the direction $\vom$ is a function of the light intensity at $P_E$, the intersection of $\vom$ and the circle $\vert \vx\vert =R$. }}
\label{figone}
\end{figure}

The geometry of the problem is  shown on Figure \ref{fig:tikz} and the data are:
\[
\Omega=\{ x:~\vert x\vert \in(R,R+H)\},\quad
Q_\nu(\vx_E,\vom)=Q^0 B_\nu(T_s)\frac{x_E^+}R.
\]
These data are used with $\vx_E=(x_E,y_E)$, the intersection of the line $\{\vx-t\vom\}_{t>0}$ with the circle $\{\vx:\vert \vx\vert =R\}$.
As $\vert \vx-t\vom\vert =R$ requires $t^2-2t\vx\cdot\vom +\vert \vx\vert ^2-R^2=0$, we have  
\[
\tau_{\vx,\vom}=\vx\cdot\vom -\sqrt{(\vx\cdot\vom )^2-\vert \vx\vert ^2+R^2}.
\]

As explained in \cite{CBOP}, for numerical convenience the problem can be rescaled so that the Planck function is $B_\nu(T)=\nu^3/(\e^{\frac\nu T}-1)$ with $T$ in Kelvin divided by $4780$. The Stefan Boltzmann formula becomes $\int_0^\infty B_\nu(T)\dd\nu =\sigma T^4 $ with $\sigma=\frac{\pi^4}{15}$.
All  cases are without scattering $a=0$. 

In the numerical tests $Q^0=5.74\cdot 10^{-5}$, $T_{sun}=1.209$, $R=0.4$, $H=0.3$.

\subsubsection{The Grey Case}

In the grey case ($\kappa_\nu$ independent of $\nu$), the upper bar, as in  $\bar J$, denotes the mean in $\nu$.  
Then it is easy to see that we need to solve iteratively the integral equation:
\begin{equation}\label{getSeconst}
\bar J(\vx)=S^E(\vx) +  \sigma Y_\kappa\star {\tilde T}^4,~ S^E(\vx)=\frac{Q^0\sigma T_s^4}{2\pi}
\int_0^{2\pi}(x- \tau_{\vx,\vom}\cos\theta)^+ 
\e^{-\kappa \tau_{\vx,\vom}}d\theta
\end{equation}
with $Y_\kappa=\ds\frac\kappa{2\pi\vert \vx\vert }\e^{-\kappa\vert \vx\vert }$.
In absence of thermal diffusion, the temperature field is given by
\begin{equation}\label{TJ}
\kappa \sigma T^4(\vx) = \kappa \bar J(\vx),~~\vx\in\Omega.
\end{equation}
Figure \ref{fig:2} shows a numerical result obtained with $\kappa=0.5$, $N=10$ iterations, starting from $T^0=0.01$.  The monotone behaviour of $\bar J^n$ is clearly seen (but not displayed here).

\begin{figure}[hbt]
\begin{minipage} [b]{0.45\textwidth}
\centering
\includegraphics[width=6.5cm]{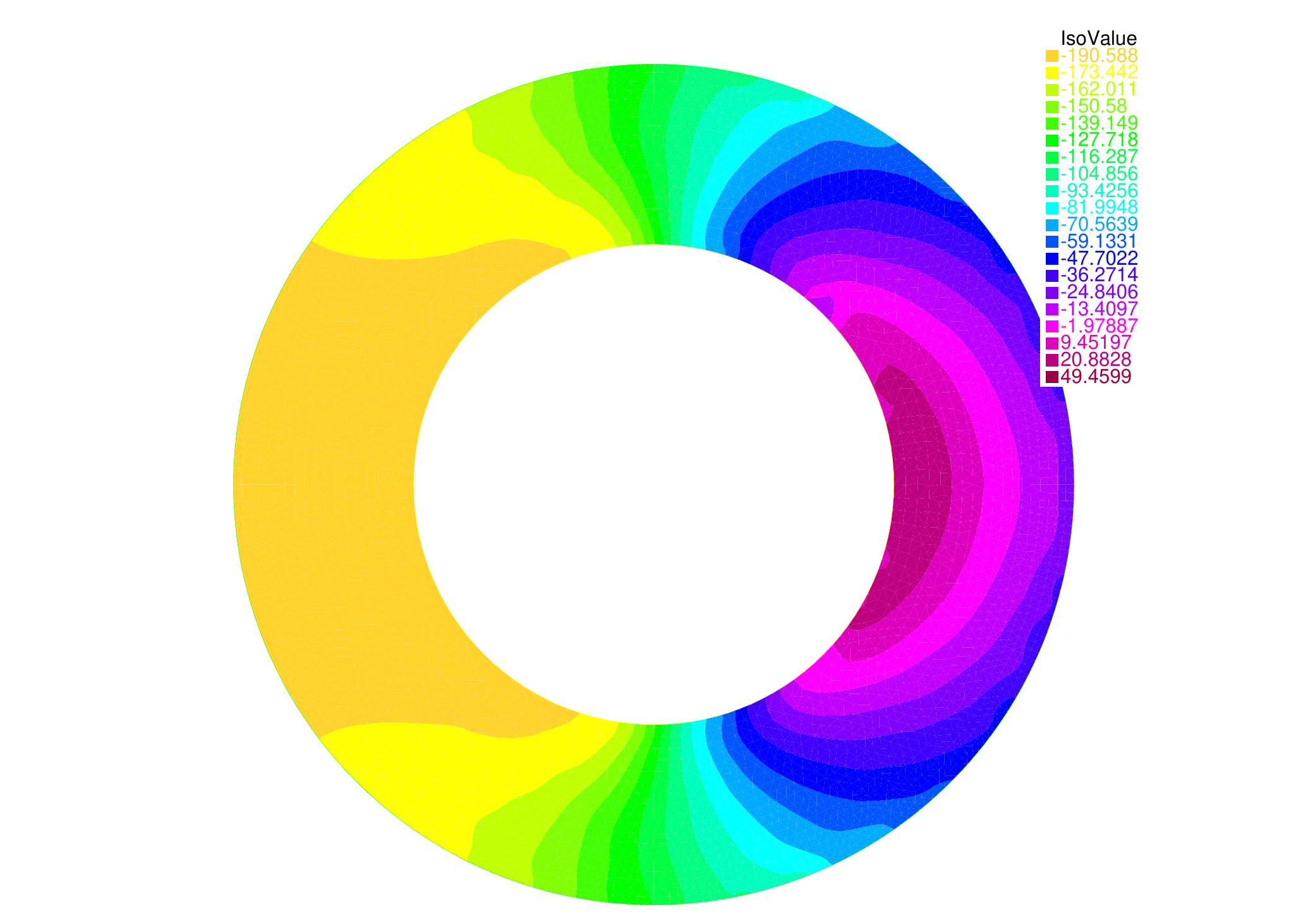}
\caption{\label{fig:2} Temperature map in the atmosphere of the planet which receives light from the right.  }
\end{minipage}
\hskip0.5cm
\begin{minipage} [b]{0.45\textwidth}
\centering
\includegraphics[width=6.5cm]{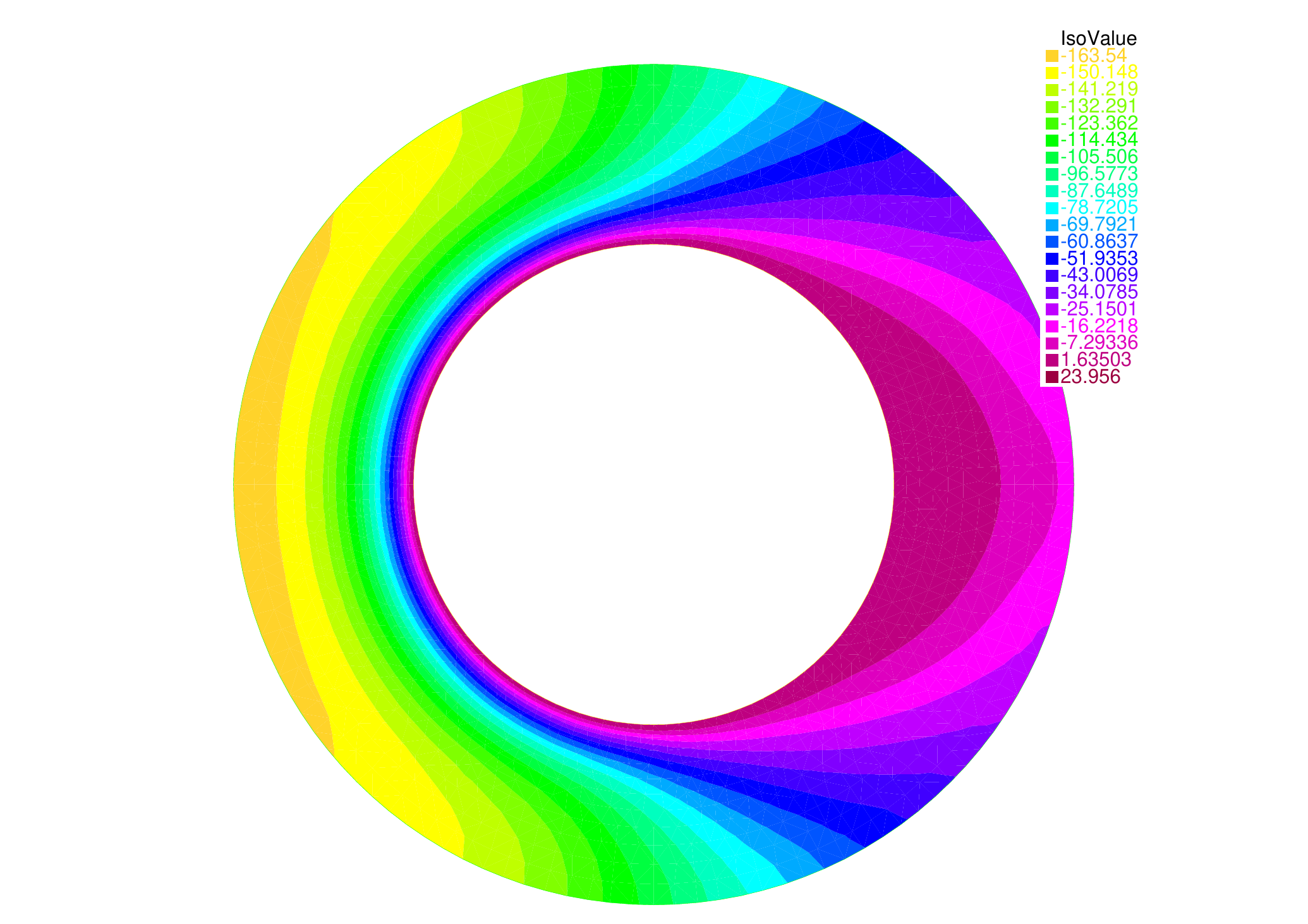}
\caption{\label{fig:3} Temperature map in the atmosphere of the planet which receives light from the right and has thermal diffusion. }
\end{minipage}
\end{figure}

The grid used for the FFT is $64\times 64$. The mesh for the ring is  $36\times 120$ approximately uniform in polar coordinates. For $S^E$ there are 60 integration points in $\theta$.   The computing time is 1 second per iteration on a core i9 MacBook 2020;  convergence is reached after 5 iterations.

\subsubsection{Non-Zero Thermal Diffusion}

Let $\kappa_T$ be the thermal diffusion and let $T_E$ be the temperature of the planet.  Then 
\eqref{TJ} must be replaced by
\begin{equation}\label{TJdiff}
-\kappa_T \Delta T + \sigma T^4(\vx) = \bar J(\vx),~~\vx\in\Omega, \quad T_{\partial\Omega}=T_E.
\end{equation}
It is discretized with triangular finite elements of degree 1 and solved iteratively by a fixed point method whereby $T^4$ is replaced by $T_m^3 T_{m+1}$.
Figure \ref{fig:3} shows a result with the same data used for Figure \ref{fig:2} and $\kappa_T=0.01\sigma$. The temperature on the planet is fixed at 0.06, i.e. 13.8 C$^o$.

\subsubsection{The Frequency Dependent Case}

When $\kappa_\nu$ is not constant the problem is numerically expensive because one Fourier transform is needed at each integration point in the integrals in $\nu$. 

Recall that, when $a_\nu=0$, we have to solve
\begin{equation}\label{BTJ}
\int_0^\infty \kappa_\nu B_\nu(T(\vx))d\nu 
 =\int_0^\infty \kappa_\nu  Y_{\kappa_\nu}\star \tilde B_\nu(T) d \nu + \bar S^E(\vx)
\end{equation}
with
\begin{equation}\label{getSe}
\bar S^E(\vx)= \frac{Q^0}{2\pi}\int_0^{2\pi}\left((x- \tau_{\vx,\vom}\cos\theta)^+
\int_0^\infty B_\nu(T_s)\kappa_\nu\e^{-\kappa_\nu \tau_{\vx,\vom}}d\nu\right)d\theta
\end{equation}

Extracting $\vx\mapsto T(\vx)$ from \eqref{BTJ},  with a known right hand side, with a $\nu\mapsto \kappa_\nu$ given by values, is doable but expensive (see \cite{FGOPSinum}).  For a simple numerical example we may expand $\kappa_\nu$ in powers of $\nu$:
\begin{equation}
\begin{aligned}&
\kappa_\nu \approx \kappa_0 + \kappa_1\nu + \kappa_2\nu^2+ \kappa_3\nu^3 + \kappa_4\nu^4+\dots
 \implies~\int_0^\infty \kappa_\nu B_\nu(T) = \sigma\kappa_0 T^4 
\\&
\hskip1cm + 24.886\kappa_1 T^5 + 122.081\kappa_2 T^6+ 726.012\kappa_3 T^7 + 5060.55 \kappa_4 T^8+\dots
\end{aligned}
\end{equation}
These numerical values are evaluations of polynomials of $\pi$ and $\zeta$ function numbers computed with Maple.

For the numerical test we chose $\kappa_\nu = \kappa_0 + \kappa_1\nu:= 0.5\pm 0.03\nu$, $\nu\in(0,15)$. Then  we have to solve iteratively
\begin{equation}\label{treize}
\begin{aligned}&
\sigma\kappa_0 T^4(\vx) +24.886\kappa_1 T^5(\vx) 
=  Y_1\star {\cal B_\nu}\vert_\vx + S^E(\vx), \quad \vx\in\Omega
\end{aligned}
\end{equation}
\begin{figure}[hbt]
\begin{minipage} [b]{0.45\textwidth}
\centering
\includegraphics[width=6.5cm]{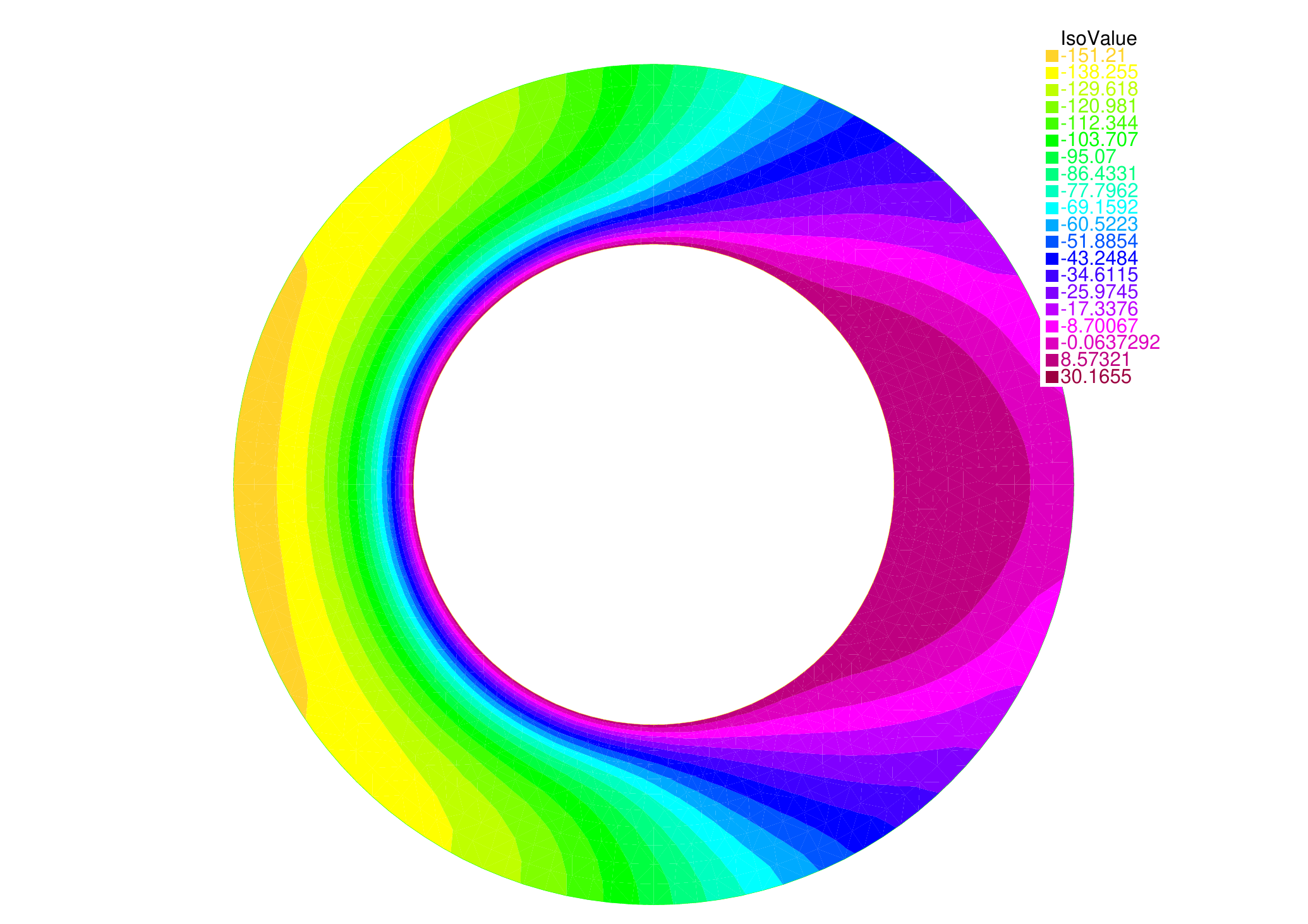}
\caption{\label{fig:4}{\it Same as above but with $\kappa=0.5+0.03\nu$, $\nu\in(0.01,15)$. }}
\end{minipage}
\hskip0.5cm
\begin{minipage} [b]{0.45\textwidth}
\centering
\includegraphics[width=6.5cm]{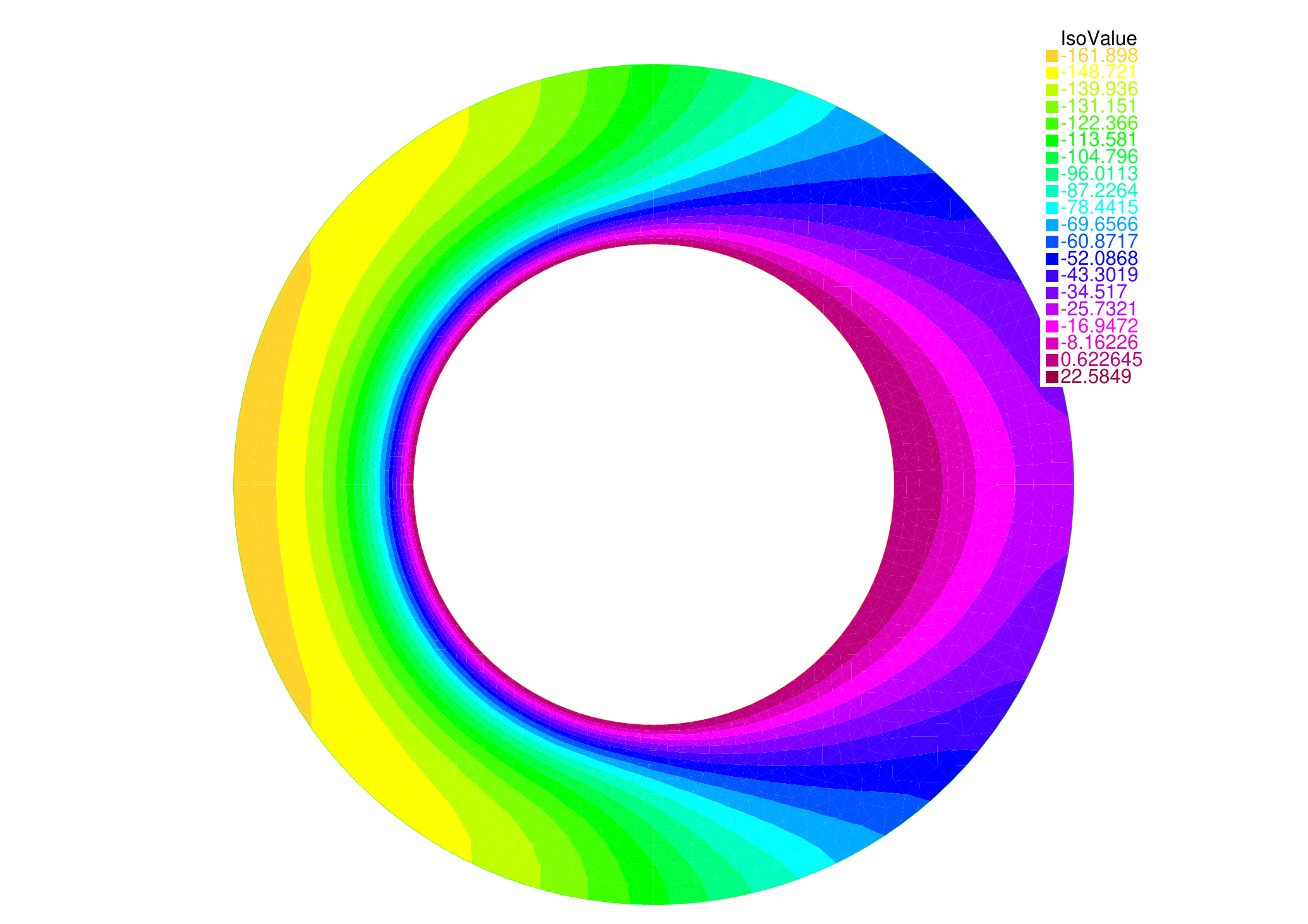}
\caption{\label{fig:5}{\it Same as above but with $\kappa=0.5-0.03\nu$, $\nu\in(0.01,15)$. }}
\end{minipage}
\end{figure}
Figures \ref{fig:3}, \ref{fig:4} and \ref{fig:5} illustrate the influence of a varying $\kappa_\nu$ on the temperature.  There were 60 points for the integrations in $\nu$, 60 points for the integrations in $\theta$ and $64\times 64$ for the Discrete Fourier Transforms.

All programs were written with the high level PDE solver \texttt{freefem++}  (see \cite{FF}).   The program for a non constant $\kappa$ is evidently much slower and took 580 seconds per case.
 
\section{Conclusion}
By using a technique developed in \cite{FGOPSinum} for the stratified radiative transfer problem, we have proved existence, uniqueness and a maximum principle for a  rather general form of the multidimensional Radiative Transfer system coupled with the time dependent temperature equation with drift. 

The proofs are constructive and  yield a robust and fairly fast  numerical numerical algorithm, at least in the grey case,  which encapsulate the exact solution between a lower and a higher numerical one obtained by starting from a guessed temperature field below (resp. above) the exact temperature field.

The 5 dimensional PDE is thus replaced by an iteration involving a three dimensional integral and a convolution integral easily computed with an FFT and which constitutes a tremendous gain in computing time over more classical finite element discretization as in \cite{HechtJCP}.

Most remarkable is that there are essentially no constraint, besides positivity, on the absorption $\kappa_\nu$ and the scattering $a_\nu$.  If these depend on $\vx$, a change of variable needs to be applied to return to the case $\kappa_\nu$ independent of $\vx$.  However if $\kappa_\nu$ depends on $T$ the method does not work, except by adding an iteration loop, sending this dependency on the right hand side of the equation of $I_\nu$.

\backmatter

\bmhead{Acknowledgments}
The authors would like to thank Prof. Claude Bardos for the numerous discussions and references given.

\bibliography{sn-bibliography}
\pagebreak
\definecolor{mGreen}{rgb}{0,0.6,0}
\definecolor{mGray}{rgb}{0.5,0.5,0.5}
\definecolor{mPurple}{rgb}{0.58,0,0.82}
\definecolor{backgroundColour}{rgb}{0.95,0.95,0.92}%

\lstdefinestyle{Cstyle}{
    backgroundcolor=\color{backgroundColour},   
    commentstyle=\color{mGreen},
    keywordstyle=\color{magenta},
    numberstyle=\tiny\color{mGray},
    stringstyle=\color{mPurple},
    basicstyle=\tiny,
    breakatwhitespace=false,         
    breaklines=true,                 
    captionpos=b,                    
    keepspaces=true,                 
    numbers=left,                    
    numbersep=5pt,                  
    showspaces=false,                
    showstringspaces=false,
    showtabs=false,                  
    tabsize=2,
    language=C
}

\section{Appendix: Code documentation}
The following may not appear in the published paper.

The following  \texttt{freefem++} script \texttt{RT2Dfull2.edp} works for $\kappa_\nu=\kappa_0+\kappa_1\nu$, $\nu\in(\nu_{min},\nu_{max})$.  It recognizes the case $\kappa$ constant (i.e. $\kappa_1=0$) and by avoiding the integrals in $\nu$ is then much faster in that case.

The data are:

 \begin{lstlisting}[style=CStyle]
 
// RT2Dfull2.edp. radiative transfer with no approximation
load "dfft"

int n=1, Niter=5,
	nx=n*32,ny=n*32,NN=nx*ny;
real Q0=2*sqrt(2.)*2.03e-5, Tsun=1.209, sigma=(pi^4)/15;

mesh Th=square(nx-1,ny-1,[-1+2*(nx-1)*x/nx,-1+2.*(ny-1)*y/ny]);
// warning  the numbering of  vertices (x,y) i  i = x/nx + nx*y/ny 

real R= 0.4, H=0.3;
real dtheta=pi/30; // controls the integration on the unit circle
real kappa0=0.5, kappa1= 0., 
	 kappaT=0.01; // if zero no heat equation
real numin= 0.01, numax = 15, dnu = (numax-numin)/100;

real source = Q0, R2=R*R; // auxiliaries
 \end{lstlisting}

We need two domains, the square for the dFFT and the ring for the physics:

\begin{lstlisting}[style=CStyle]

mesh Th=square(nx-1,ny-1,[-1+2*(nx-1)*x/nx,-1+2.*(ny-1)*y/ny]);
// warning  the numbering of  vertices (x,y) i  i = x/nx + nx*y/ny 
border R1(t=0,2*pi){x=R*cos(t); y=R*sin(t);}
border RH(t=0,2*pi){x=(R+H)*cos(t); y=(R+H)*sin(t);}
mesh Rh= buildmesh(RH(120)+R1(-120));
 \end{lstlisting}

The finite element spaces and the functions are defined by

\begin{lstlisting}[style=CStyle]

fespace Vh(Th,P1); 
fespace Wh(Rh,P1);

Vh<complex> u,v,w, F; // u,v,w for FFT and JJ for J(x,y)
Vh Jsource; // for S^E(x,y)
Wh Tc=0.01, Tch, Tca; // Tc is T(x,y) and Tch and Tca are auxiliaries
 \end{lstlisting}
We need a function to define $\kappa_\nu$, one to define the Planck function $B_\nu(T)$ and one to compute $\tau_{\vx,\vom}$. The parameter \texttt{scal} in \texttt{twx} is here to save time and prevent recomputing the scalar product in \texttt{getSe}.

\begin{lstlisting}[style=CStyle]

func real kappa(real nu) {return kappa0+kappa1*nu;} 
func real Bnu(real T,real nu){ return sqr(nu)*nu/(exp(nu/T)-1);}
func real txw(real X,real Y, real scal){
	real aux =  sqr(scal) + R2 -sqr(X)-sqr(Y);
	if(aux>=0) 
		if(scal>0) return scal - sqrt(aux);
			else return scal + sqrt(aux);
	else return -1;
}
 \end{lstlisting}

Now \texttt{getSe} implements \eqref{getSe} or \eqref{getSeconst} when applicable.

\begin{lstlisting}[style=CStyle]

func real getSe(real X, real Y){
	real aux = X*X+Y*Y;
	real Jxy=0;
	if(aux>R2 && aux<(R+H)*(R+H))
	  for(real theta=0; theta<2*pi;theta+=dtheta){
		real wx=cos(theta),wy=sin(theta);
		real scal  = X*wx+Y*wy;
		real t = txw(X,Y,scal);
		real Bke = 0;
		if(t>0) {
			if(kappa1==0) Bke = kappa0*exp(-t*kappa0)*sigma*Tsun^4;
			else for(real nu=numin;nu<numax;nu+=dnu)
					Bke += Bnu(Tsun,nu)*kappa(nu)*exp(-kappa(nu)*t)*dnu;
			Jxy=Jxy+dtheta*max(X-t*wx,0.)*Bke;
		}
	 }
	return Jxy;
}
Jsource = getSe(x,y)*source/(2*pi);
 \end{lstlisting}
 
 Now let us compute $Y_1$ and its Fourier transform \texttt{v}. We could have use its analytical values but then would have had to struggle with the correspondance between the Fourier modes and the grid points. To avoid the singularity at $\vx=0$  we truncate it at $\vert\vx\vert>R/4$. The FEM function $u$ is needed to build an array of values at the grid points.

\begin{lstlisting}[style=CStyle]

func real Yxy(real X,real Y,real kappa){
	real aux = sqrt(X*X+Y*Y);
	if( aux>R/4 ) return kappa*exp(-aux*kappa)/(2*pi*aux);
	else return 0.;
}
 \end{lstlisting}
 
 The computation of the right hand side of \eqref{treize} is done as follows
 
\begin{lstlisting}[style=CStyle]
	F=0;
	if(kappa1==0){
		 if(niter==0){u = Yxy(x,y,kappa0); v[]=dfft(u[],ny,-1);}
		 u = kappa0*sigma*Tc^4;
		 w[]=dfft(u[],ny,-1);
		 F=v*w*kappa0/sqr(NN);		 
	} else
	for(real nu=numin;nu<numax;nu+=dnu){
		u= Yxy(x,y,kappa(nu));
		v[]=dfft(u[],ny,-1);
		u = Bnu(Tc(x,y),nu);
		w[]=dfft(u[],ny,-1);
		u=v*w/sqr(NN);
		F=F+u*kappa(nu)*dnu;	
	}
	u[]=dfft(F[],ny,1);
	u= u + Jsource;
 \end{lstlisting}
 
 Finally the temperature is computed and converted into Celsius degree by the last formula.
 
\begin{lstlisting}[style=CStyle]
	Tca=sqrt(sqrt(real(u) / (sigma*kappa0 + 24.886*kappa1*Tc) ));
	heat; 
	u = sqr(Tc*Tc)*(sigma*kappa0 + 24.886*kappa1*Tc); // mean light intensity
	Tca=Tc*4780-273; // temperature in Celcius
	plot(Tca,ps="planettempdifffull2.ps", value=1,fill=1);
 \end{lstlisting}
 
 where \texttt{heat} is finite element solver for the temperature equation implemented as
 (notice how $T$ is prescribed on the planet at 0.06, which is $13.8$ Celsius.
\begin{lstlisting}[style=CStyle]
problem   heat(Tc,Tch) = int2d(Rh)(kappaT*(dx(Tc)*dx(Tch)
			+dy(Tc)*dy(Tch)) +Tc*Tch) - int2d(Rh)(Tca*Tch)+on(R1,Tc=0.06);
 \end{lstlisting}
 
 These next to last 2 blocks must be encapsulated into a iteration loop
\begin{lstlisting}[style=CStyle]

for(int niter=0;niter<Niter; niter++){
// the blocks here
}
 \end{lstlisting}

\end{document}